\documentclass[11pt]{amsart}

\usepackage{graphicx}
\usepackage{amsmath,amsfonts,amssymb,mathabx}

\newtheorem{theorem}{Theorem}[section]
\newtheorem{lemma}[theorem]{Lemma}
\newtheorem{proposition}[theorem]{Proposition}
\newtheorem{corollary}[theorem]{Corollary}

\theoremstyle{definition}
\newtheorem{definition}[theorem]{Definition}

\theoremstyle{remark}
\newtheorem{remark}[theorem]{Remark}

\numberwithin{equation}{section}

\newcommand{\inv}{\ensuremath\mathrm{inv}}
\newcommand{\coinv}{\ensuremath\mathrm{coinv}}
\newcommand{\Des}{\ensuremath\mathrm{Des}}
\newcommand{\maj}{\ensuremath\mathrm{maj}}
\newcommand{\comaj}{\ensuremath\mathrm{comaj}}
\newcommand{\des}{\ensuremath\mathrm{des}}
\newcommand{\wt}{\ensuremath\mathrm{wt}}

\newcommand{\SYD}{\ensuremath\mathrm{SYD}}

\newcommand{\SSYT}{\ensuremath\mathrm{SSYT}}
\newcommand{\SYT}{\ensuremath\mathrm{SYT}}

\newcommand{\SSKD}{\ensuremath\mathrm{SSKD}}
\newcommand{\SKD}{\ensuremath\mathrm{SKD}}
\newcommand{\YKD}{\ensuremath\mathrm{YKD}}

\newcommand{\SSKT}{\ensuremath\mathrm{SSKT}}
\newcommand{\SKT}{\ensuremath\mathrm{SKT}}

\newcommand{\std}{\ensuremath\mathrm{std}}

\newcommand{\swap}{\ensuremath\mathfrak{s}}
\newcommand{\braid}{\ensuremath\mathfrak{b}}

\newcommand{\key}{\ensuremath\kappa}
\newcommand{\fund}{\ensuremath\mathfrak{F}}
\newcommand{\mac}{\ensuremath{E}}
\newcommand{\Mac}{\ensuremath{H}}

\newlength\cellsize 

\savebox2{%
\begin{picture}(12,12)
\put(0,0){\line(1,0){12}}
\put(0,0){\line(0,1){12}}
\put(12,0){\line(0,1){12}}
\put(0,12){\line(1,0){12}}
\end{picture}}

\newcommand\cellify[1]{\def\thearg{#1}\def\nothing{}%
\ifx\thearg\nothing\vrule width0pt height\cellsize depth0pt%
  \else\hbox to 0pt{\usebox2\hss}\fi%
  \vbox to 12\unitlength{\vss\hbox to 12\unitlength{\hss$#1$\hss}\vss}}

\newcommand\tableau[1]{\vtop{\let\\=\cr
\setlength\baselineskip{-12000pt}
\setlength\lineskiplimit{12000pt}
\setlength\lineskip{0pt}
\halign{&\cellify{##}\cr#1\crcr}}}

\begin{document}


\title[Nonsymmetric Macdonald polynomials]{Nonsymmetric Macdonald polynomials and \\ a refinement of Kostka--Foulkes polynomials}  

\author[S. Assaf]{Sami Assaf}
\address{Department of Mathematics, University of Southern California, Los Angeles, CA 90089}
\email{shassaf@usc.edu}






\begin{abstract}
  We study the specialization of the type A nonsymmetric Macdonald polynomials at $t=0$ based on the combinatorial formula of Haglund, Haiman, and Loehr. We prove that this specialization expands nonnegatively into the fundamental slide polynomials, introduced by the author and Searles. Using this and weak dual equivalence, we prove combinatorially that this specialization is a positive graded sum of Demazure characters. We use stability results for fundamental slide polynomials to show that this specialization stabilizes and to show that the Demazure character coefficients give a refinement of the Kostka--Foulkes polynomials.
\end{abstract}

\maketitle

%
\section{Introduction}
%
\label{sec:introduction}

Macdonald's symmetric functions \cite{Mac88} are two parameter generalizations of classical symmetric functions \cite{Mac95} that simultaneously generalize the Hall--Littlewood symmetric functions and Jack symmetric functions. The transformed Macdonald symmetric functions in type A, commonly denoted by $\Mac_{\mu}(X;q,t)$, are known to have deep connections to representation theory of the symmetric group as shown by Garsia and Haiman \cite{GH96} and geometry of Hilbert schemes as shown by Haiman \cite{Hai01}. In 2004, Haglund \cite{Hag04} gave an elegant combinatorial formula for the monomial expansion of Macdonald symmetric functions, and Haglund, Haiman, and Loehr \cite{HHL05} proved and generalized this formula to include Macdonald integral forms $J_{\mu}(X;q,t)$, which are obtained from Macdonald's original orthogonal polynomials $P_{\mu}(X;q,t)$ by a scalar multiple.

The nonsymmetric Macdonald polynomials were introduced by Opdam \cite{Opd95} and Macdonald \cite{Mac96}, and generalized by Cherednik \cite{Che95}. Results in this nonsymmetric setting often extend to any root system giving hope that by passing through the nonsymmetric variations, one might be able to shed more light on the symmetric Macdonald polynomials in general types. Generalizing \cite{HHL05}, Haglund, Haiman and Loehr \cite{HHL08} gave a combinatorial formula for the nonsymmetric Macdonald polynomials in type A, commonly denoted by $\mac_a(X;q,t)$. They prove, combinatorially, that the nonsymmetric $\mac_a(X;q,t)$ stabilizes to the symmetric $P_{\mu}(X;q,t)$, emulating a similar result of Knop and Sahi \cite{KS97} for Jack polynomials, by relating the combinatorial models for both.

The connection between specializations of Macdonald polynomials and Demazure characters began with Sanderson \cite{San00} who used the theory of nonsymmetric Macdonald polynomials in type A to construct an affine Demazure module with graded character $P_{\mu}(X;q,0)$, similar to the construction of Garsia and Procesi \cite{GP92} for Hall-Littlewood symmetric functions $\Mac_{\mu}(X;0,t)$. Ion \cite{Ion03} generalized this result to nonsymmetric Macdonald polynomials in general type using the method of intertwiners in double affine Hecke algebras to realize $\mac_{a}(X;q,0)$ as an affine Demazure character. He also showed that $\mac_{a}(X;0,0)$ is a (finite) Demazure character.

To clarify connections with other specializations, the specialization $\mac_{a}(X;0,t)$ was studied by Ion \cite{Ion08} and also by Descouens and Lascoux \cite{DL05} who dubbed these nonsymmetric Hall--Littlewood polynomials since they stabilize precisely to the Hall--Littlewood polynomials $P_{\mu}(x;0,t)$. The specialization we consider in this paper, that of $\mac_{a}(X;q,0)$, stabilizes to $\omega P_{\mu}(x;0,t)$, where $\omega$ is the well-known involution on symmetric functions. That is, for $a$ weakly increasing, $\mac_{a}(X;0,1)$ is a homogeneous symmetric polynomial whereas $\mac_{a}(X;1,0)$ is an elementary symmetric polynomial.

We mention one further specialization, that of $\mac_{a}(X;\infty,\infty)$. Ion \cite{Ion08}, again using Hecke algebras, showed that this is precisely a Demazure atom, and this was also proved combinatorially by Haglund, Haiman, and Loehr \cite{HHL08} in type A. This specialization was studied further by Mason \cite{Mas09} who developed a combinatorial theory parallel to that for Schur functions. 

In this paper, we consider the specialization $\mac_a(X;q,0)$ and relate this to the (finite) type A Demazure characters $\mac_a(X;0,0)$. This result suggest that an analogous statement might hold in other types as well. For our purposes, since Demazure characters stabilize to Schur functions, this paper provides a bridge between the combinatorics of the symmetric and nonsymmetric settings by drawing direct parallels between expansions and specializations on both sides.

To begin, we show that the expansion of $\mac_a(X;q,0)$ into fundamental slide polynomials $\fund_b$, introduced by Assaf and Searles \cite{AS17} to study Schubert polynomials, is a polynomial in $q$ with nonnegative integer coefficients. This parallels the expansion of $\Mac_{\mu}(X;q,t)$ into fundamental quasisymmetric functions $F_{\beta}$, introduced by Gessel \cite{Ges84}. Assaf and Searles showed that the fundamental slide polynomials stabilize to fundamental quasisymmetric functions, so together this recovers the stability results of \cite{HHL08,KS97}. 

Utilizing the theory of weak dual equivalence \cite{Ass-1}, we group together terms in the fundamental slide expansion of $\mac_a(X;q,0)$ to prove, combinatorially, that the coefficients of $\mac_a(X;q,0)$ when expanded into Demazure characters are polynomials in $q$ with nonnegative integer coefficients. This parallels the use of dual equivalence \cite{Ass15} which collects terms in the fundamental quasisymmetric expansion of $\Mac_{\mu}(X;q,t)$ into classes that are conjecturally Schur positive \cite{Ass15}, and were proved to be Schur positive for the case of $\Mac_{\mu}(X;0,t)$ by Roberts \cite{Rob-un}. Moreover, we prove that the involutions that group terms commute with the combinatorial bijection that proves stability.

Finally, we interpret our results to give a refinement of the \emph{Kostka--Foulkes polynomials} $K_{\lambda,\mu}(t)$ that give the change of basis coefficients from $\Mac_{\mu}(X;0,t)$ to Schur functions in terms of the \emph{nonsymmetric Kostka--Foulkes polynomials} $K_{a,b}(q)$ that similarly give the change of basis coefficients from $\mac_{b}(X;q,0)$ to Demazure characters.

Our proofs are purely combinatorial, and the paper is largely self-contained. The main hurdle in pushing this work further is the lack of known (or even conjectured) positivity for $\mac_a(X;q,t)$. Indeed, the Schur positivity for $\Mac_{\mu}(X;q,t)$, conjectured by Macdonald \cite{Mac88} and proved by Haiman \cite{Hai01}, translates to Macdonald's original $P_{\mu}(X;q,t)$ and the integral form $J_{\mu}(X;q,t)$ via \emph{plethystic substitution}. There is no known analog of plethysm for the full polynomial ring.

\begin{center}
{\sc Acknowledgments}
\end{center}

The author thanks Per Alexandersson, who together with Sawhney is studying the same specialization \cite{AS-un}, for sharing data suggesting that the specialization of nonsymmetric Macdonald polynomials considered in this paper appear to expand nonnegatively into key polynomials. The author is also grateful to Jim Haglund and Bogdan Ion for detailed conversations about nonsymmetric Macdonald polynomials and their specializations.

%
\section{Demazure characters}
%
\label{sec:key}

Throughout we let $X$ denote the finite variable set $x_1, x_2, \ldots, x_n$. Polynomials in $n$ variables are naturally index by weak compositions of length at most $n$. The \emph{key diagram} of a weak composition $a$ is the collection of cells in the $\mathbb{N}\times\mathbb{N}$ lattice with $a_i$ cells left-justified in row $i$. Key diagrams play the analogous role for polynomials that Young diagrams play for symmetric functions. 

Standard key tableaux were introduced by Assaf \cite{Ass-1} to develop a theory of type A Demazure characters, which we call \emph{key polynomials}, parallel to that of Schur functions. 

\begin{definition}
  A \emph{key tableau} is a filling of a key diagram with positive integers such that columns have distinct entries, rows weakly decrease, and, if some entry $i$ is above and in the same column as an entry $k$ with $i<k$, then there is an entry immediately right of $k$, say $j$, and $i<j$.
  \label{def:key-tableau}
\end{definition}

A \emph{standard key tableau} is a bijective filling of a key diagram. This coincides precisely with the definition in \cite{Ass-1} since standard key tableaux necessarily have strictly decreasing rows. For examples, see Figure~\ref{fig:SKT}. Denote the set of standard key tableaux of shape $a$ by $\SKT(a)$. 
  
\begin{figure}[ht]
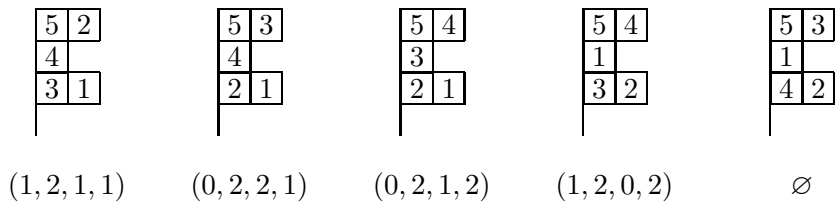

  \begin{displaymath}
    \begin{array}{c@{\hskip 2\cellsize}c@{\hskip 2\cellsize}c@{\hskip 2\cellsize}c@{\hskip 3\cellsize}c}
      \vline\tableau{5 & 2 \\ 4 \\ 3 & 1 \\ & } & 
      \vline\tableau{5 & 3 \\ 4 \\ 2 & 1 \\ & } & 
      \vline\tableau{5 & 4 \\ 3 \\ 2 & 1 \\ & } & 
      \vline\tableau{5 & 4 \\ 1 \\ 3 & 2 \\ & } & 
      \vline\tableau{5 & 3 \\ 1 \\ 4 & 2 \\ & } \\ \\
      (1,2,1,1) & (0,2,2,1) & (0,2,1,2) & (1,2,0,2) & \varnothing
    \end{array}
  \end{displaymath}
  \caption{\label{fig:SKT}Standard key tableaux of shape $(0,2,1,2)$ and their weak descent compositions.}
\end{figure}

The \emph{key polynomials} first arose as Demazure characters for type A \cite{Dem74} and were later studied combinatorially by Lascoux and Sch{\"u}tzenberger \cite{LS90}. Reiner and Shimozono \cite{RS95} gave a thorough treatment of key polynomials, including four equivalent definitions. We choose to begin with a different definition, due to Assaf \cite{Ass-1}, but we postpone this to introduce another basis for polynomials that will facilitate our main results.

Assaf and Searles \cite{AS17} introduced the fundamental slide basis for polynomials that parallels Gessel's fundamental basis for quasisymmetric functions \cite{Ges84}. Given two weak compositions $a$ and $b$ of length $n$, write $b \geq a$ if $b_1 + \cdots + b_k \geq a_1 + \cdots + a_k$ for all $k=1,\ldots,n$. Given compositions $\alpha,\beta$, write \emph{$\beta$ refines $\alpha$} if there exist indices $i_1<\ldots<i_k$ such that $\beta_1 + \cdots + \beta_{i_j} = \alpha_1 + \cdots + \alpha_j$. For example, $(1,2,2)$ refines $(3,2)$ but does not refine $(2,3)$.

\begin{definition}[\cite{AS17}]
  The \emph{fundamental slide polynomial} $\fund_{a}(X)$ is given by
  \begin{equation}
    \fund_{a}(X) = \sum_{\substack{b \geq a \\ \mathrm{flat}(b) \ \mathrm{refines} \ \mathrm{flat}(a)}} X^b,
    \label{e:slide}
  \end{equation}
  where $\mathrm{flat}(a)$ is the composition obtained by removing zero parts from $a$.
\end{definition}

For example, we compute
\[ \fund_{(0,2,1,2)}(x_1,x_2,x_3,x_4) = x_2^2 x_3 x_4^2 + x_1 x_2 x_3 x_4^2 + x_1^2 x_3 x_4^2 + x_1^2 x_2 x_4^2 + x_1^2 x_2 x_3 x_4 + x_1^2 x_2 x_3^2 .\]

We use the fundamental slide polynomials to define key polynomials as the generating polynomials for standard key tableaux. To do so, we assign to each standard key tableau (more generally, to each standard filling of a key diagram), a weak composition.

\begin{definition}
  Given a standard filling $T$ of a key diagram, the \emph{weak descent composition of $T$}, denoted by $\des(T)$, is constructed as follows. Partition the decreasing permutation $n \cdots 2 1$ into blocks, say $\tau^{(k)} | \cdots | \tau^{(1)}$, broken between $i+1$ and $i$ precisely when $i+1$ lies weakly right of $i$ in $T$. Set $t_k$ to be the row of $\tau^{(k)}_1$ if it lies in the first column and otherwise $n$. For $i<k$, set $t_i = \min(\mathrm{row}(\tau^{(i)}_1),t_{i+1}-1)$ if $\tau^{(i)}_1$ lies in the first column and otherwise $t_{i+1}-1$. Set $\des(T)_{t_i} = |\tau^{(i)}|$ and all other parts are zero if all $t_i>0$ and $\des(T) = \varnothing$ otherwise. 
  \label{def:weak-des}
\end{definition}

\begin{remark}
  In \cite{Ass-1}, the weak descent composition of $T$ is computed without the first column caveat. However, when $T$ is a standard key tableau, rows weakly decrease, so any entry not in the first column that occurs as $\tau^{(i)}_1$ will necessarily result in $t_i = t_{i+1}-1$ since the entry to its left must lie in some previous block. Thus the definitions agree when $T$ is a standard key tableau.
\end{remark}

For example, the weak descent compositions for $\SKT(0,2,1,2)$ are given in Figure~\ref{fig:SKT}.


If $\des(T) = \varnothing$, then we say that $T$ is \emph{virtual}. Extend previous notation for fundamental slide polynomials to avoid discounting virtual objects by setting 
\begin{equation}
  \fund_{\varnothing} = 0.
\end{equation}

We have the following expansion for key polynomials in terms of fundamental slide polynomials that we may take as our definition.

\begin{proposition}[\cite{Ass-1}]
  The key polynomial $\key_a(X)$ is given by
  \begin{equation}
    \key_{a}(X) = \sum_{T \in \SKT(a)} \fund_{\des(T)}(X) .
  \end{equation}
  \label{prop:key}
\end{proposition}

For example, from Figure~\ref{fig:SKT} we compute
\[ \key_{(0,2,1,2)} =  \fund_{(1,2,1,1)} + \fund_{(0,2,2,1)} + \fund_{(0,2,1,2)} + \fund_{(1,2,0,2)}. \]

A \emph{semi-standard key tableau} is a key tableau in which no entry exceeds its row index. Denote the set of semi-standard key tableaux of shape $a$ by $\SSKT(a)$. For example, see Figure~\ref{fig:SSKT}.

\begin{figure}[ht]
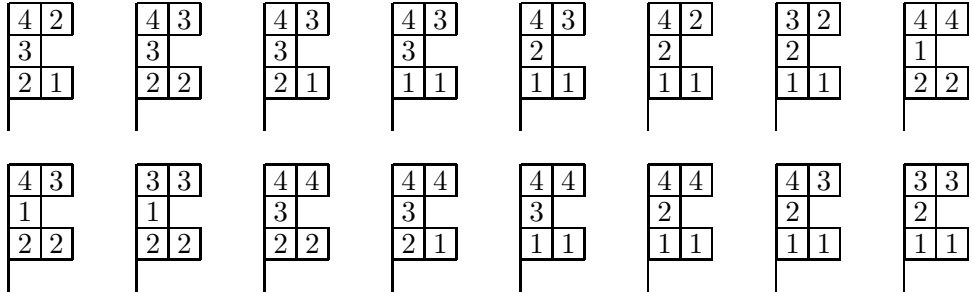

  \begin{displaymath}
    \begin{array}{c@{\hskip 2\cellsize}c@{\hskip 2\cellsize}c@{\hskip 2\cellsize}c@{\hskip 2\cellsize}c@{\hskip 2\cellsize}c@{\hskip 2\cellsize}c@{\hskip 2\cellsize}c}
      \vline\tableau{4 & 2 \\ 3 \\ 2 & 1 \\ & } & 
      \vline\tableau{4 & 3 \\ 3 \\ 2 & 2 \\ & } & 
      \vline\tableau{4 & 3 \\ 3 \\ 2 & 1 \\ & } & 
      \vline\tableau{4 & 3 \\ 3 \\ 1 & 1 \\ & } & 
      \vline\tableau{4 & 3 \\ 2 \\ 1 & 1 \\ & } & 
      \vline\tableau{4 & 2 \\ 2 \\ 1 & 1 \\ & } & 
      \vline\tableau{3 & 2 \\ 2 \\ 1 & 1 \\ & } & 
      \vline\tableau{4 & 4 \\ 1 \\ 2 & 2 \\ & } \\ \\
      \vline\tableau{4 & 3 \\ 1 \\ 2 & 2 \\ & } & 
      \vline\tableau{3 & 3 \\ 1 \\ 2 & 2 \\ & } & 
      \vline\tableau{4 & 4 \\ 3 \\ 2 & 2 \\ & } & 
      \vline\tableau{4 & 4 \\ 3 \\ 2 & 1 \\ & } & 
      \vline\tableau{4 & 4 \\ 3 \\ 1 & 1 \\ & } & 
      \vline\tableau{4 & 4 \\ 2 \\ 1 & 1 \\ & } & 
      \vline\tableau{4 & 3 \\ 2 \\ 1 & 1 \\ & } & 
      \vline\tableau{3 & 3 \\ 2 \\ 1 & 1 \\ & } 
    \end{array}
  \end{displaymath}
  \caption{\label{fig:SSKT}The semi-standard key tableaux of shape $(0,2,1,2)$.}
\end{figure}

\begin{proposition}
  The \emph{key polynomial} $\key_a(X)$ is given by
  \begin{equation}
    \key_{a}(X) = \sum_{T \in \SSKT(a)} X^{\wt(T)} ,
    \label{e:key}
  \end{equation}
  where $\wt(T)$ is the weak composition whose $i$th part is the number of entries equal to $i$.
  \label{prop:key-std}
\end{proposition}

\begin{proof}
  We may define a \emph{standardization map} from semi-standard key tableaux to standard key tableaux as follows. Given $T \in \SSKT(a)$, relabel the cells of $T$ from $1$ to $n$ in the following order: for $k$ from $1$ to $n$, relabel cells labeled $k$ from right to left in $T$. Since columns have distinct values, this is well-defined and necessarily results in a filling with distinct column entries and strictly decreasing rows. Moreover, it clearly preserves the property that if some entry $i$ is above and in the same column as an entry $k$ with $i<k$, then there is an entry immediately right of $k$, say $j$, and $i<j$. In particular, the result, denoted by $\std(T)$, lies in $\SKT(a)$.

  If $\std(T)=S$, then by construction $\wt(T)$ refines $\des(S)$. Conversely, we claim that given $S \in \SKT(a)$, for every weak composition $b$ such that $b\geq\des(S)$ and $\mathrm{flat}(b)$ refines $\mathrm{flat}(\des(S))$ as compositions, there is a unique $T \in \SSKT(a)$ with $\wt(T) = b$ such that $\std(T) = S$. From the claim, for $S \in \SKT(a)$, we have
  \begin{displaymath}
    \sum_{T \in \std^{-1}(S)} X^{\wt(T)} = \fund_{\des(S)}(X),
  \end{displaymath}
  and the result follows. To construct $T$ from $b$ and $S$, for $j = n,\ldots,1$, if $\des(S)_{j} = b_{i_{j-1} + 1} + \cdots + b_{i_{j}}$, then, from left to right, change each of the first $b_{i_{j-1} + 1}$ $j$'s to $i_{j-1} + 1$, the next $b_{i_{j-1} + 2}$ $j$'s to $i_{j-1} + 2$, and so on. Existence is proved, and uniqueness follows from the lack of choice. 
\end{proof}

%
\section{Nonsymmetric Macdonald polynomials}
%
\label{sec:hall}


Given a weak composition $a$, two cells of the key diagram for $a$ are \emph{attacking} if they lie in the same column or if they lie in adjacent columns with the cell on the left strictly higher than the cell on the right. A filling is \emph{non-attacking} if no two attacking cells have the same value and no cell in the first column exceeds its row index. Note that this latter condition is equivalent to adding a \emph{basement} in the sense of Haglund, Haiman, and Loehr \cite{HHL08}.

For example, the filling in Figure~\ref{fig:key-fill} is non-attacking as are the $\SSKT(0,2,1,2)$ in Figure~\ref{fig:SSKT}.

\begin{figure}[ht]
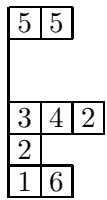

  \begin{displaymath}
    \vline\tableau{ 5 & 5 \\ \\ \\ 3 & 4 & 2 \\ 2 \\ 1 & 6 }
  \end{displaymath}
  \caption{\label{fig:key-fill}A non-attacking filling of the key diagram for $(2,1,3,0,0,2)$}
\end{figure}

The \emph{leg} of a cell of a key diagram is the number of cells weakly to its right. Given a non-attacking filling $T$, define $\maj(T)$ to be the sum of the legs of all cells $c$ such that the entry in $c$ is strictly greater than the entry immediately to its left. For example, the filling in Figure~\ref{fig:key-fill} has $\maj = 3$.

A \emph{triple} of a key diagram is a collection of three cells with two row adjacent and either (Type I) the third cell is above the left and the lower row is strictly longer, or (Type II) the third cell is below the right and the higher row is weakly longer. The \emph{orientation} of a triple is determined by reading the entries of the cells from smallest to largest. A \emph{co-inversion triple} is a Type I triple oriented clockwise or a Type II triple oriented counterclockwise. For an illustration, see Figure~\ref{fig:key-triple}. Given a non-attacking filling $T$, define $\coinv(T)$ to be the number of co-inversion triples of $T$. For example, the filling in Figure~\ref{fig:key-fill} has $\coinv = 2$.

\begin{figure}[ht]
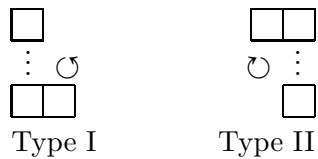

  \begin{displaymath}
    \begin{array}{l}
      \tableau{ \ } \\[-0.5\cellsize] \hspace{0.4\cellsize} \vdots \hspace{0.5\cellsize} \circlearrowleft \\ \tableau{ \ & \ } \\ \mbox{Type I}
    \end{array}
    \hspace{3\cellsize}
    \begin{array}{r}
      \tableau{ \ & \ } \\[-0.5\cellsize] \circlearrowright \hspace{0.5\cellsize} \vdots \hspace{0.4\cellsize} \\ \tableau{ \ } \\ \mbox{Type II}
    \end{array}
  \end{displaymath}
  \caption{\label{fig:key-triple}Co-inversion triples for key diagrams.}
\end{figure}

\begin{proposition}
  The set of non-attacking fillings of $a$ with $\maj(T)=\coinv(T)=0$ is $\SSKT(a)$.
  \label{prop:SSKT-D}
\end{proposition}

\begin{proof}
  Suppose $T\in\SSKT(a)$. By definition, the first column of $T$ must have entries weakly smaller than the row index. Since columns are distinct, no attacking cells in the same column have the same values. Suppose cells $c,d$ are attacking, say with $c$ strictly above and in the column immediately left of $d$, and have the same value, say $i$. Since rows of $T$ weakly decrease, the entry immediately left of $d$ must be $k$ for some $k \geq i$. Since columns are distinct, we must have $k>i$. Then the condition on key tableaux forces $d$ to have entry greater than $i$, a contraction. Therefore $T$ is non-attacking. Since rows weakly decrease, $\maj(T)=0$. Given a Type I triple, say with $i$ in the top row and $k$ left of $j$ in the bottom row, we must have $k\geq j$ since rows weakly decrease, and the key tableaux condition forces either $i>k\geq j$ or $k \geq j > i$, in which case the orientation is clockwise and the triple is not a co-inversion triple. Given a Type II triple, say with $k$ left of $j$ in the top row and $i$ in the bottom row, we must have $k \geq j$. If $j>i$, then the orientation is counter-clockwise and this is not a co-inversion triple, and if $j<i$ then the lower row must be weakly longer than the higher, so this is not a triple. Therefore $\coinv(T)=0$ as well.

  Conversely, suppose $T$ is a non-attacking filling with $\maj(T)=\coinv(T)=0$. Since $\maj(T)=0$, rows weakly decrease. Since $T$ is non-attacking, columns are distinct and entries in the first column may not exceed their row index. Again using weakly decreasing rows, this means no entry in a row may exceed its row index. Finally, suppose that some entry $i$ is above and in the same column as an entry $k$ with $i<k$ and consider the leftmost such pair. If there is no entry immediately right of $k$, then $i$ and $k$ form the right column of a Type II inversion triple since the row of $i$ must be weakly longer than the row of $k$ and weakly decreasing rows ensures the entry left of $i$ (or its row index if it is in the first column) is greater than $i$ and the leftmost assumption ensure that is greater than the entry left of $k$, which is weakly greater than $k$. If there is an entry right of $k$, say $j$, and $i \geq j$, then these three cells form a Type I co-inversion triple. Therefore the key tableaux condition holds, so $T\in\SSKT(a)$.
\end{proof}

Haglund, Haiman and Loehr \cite{HHL08} proved that the nonsymmetric Macdonald polynomial $\mac_{a}(X;q,t)$ is the generating polynomial of non-attacking fillings of the key diagram for $a$ $q$-counted by $\maj$ and $t$-counted by $\coinv$ with an additional product term that collapses to $1$ when $t=0$.

\begin{definition}[\cite{HHL08}]
  The nonsymmetric Macdonald polynomial $\mac_{a}(X;q,t)$ is given by
  \begin{equation}
    \mac_{a}(X;q,t) = \sum_{\substack{T:a\rightarrow [n] \\ \mathrm{non-attacking}}} q^{\maj(T)} t^{\coinv(T)} X^{\wt(T)}
    \prod_{c \neq \mathrm{left}(c)} \frac{1-t}{1 - q^{\mathrm{leg}(c)+1} t^{\mathrm{arm}(c)+1}},
  \end{equation}
  where the product is over cells of the key diagram such that the cell to its left (or the row index) has a different entry, and $\mathrm{arm}(c)$ is the number of cells below $c$ in a weakly shorter row or above $c$ and one column to the left in a strictly shorter row.
  \label{def:mac}
\end{definition}

By Proposition~\ref{prop:SSKT-D}, we have a simple combinatorial proof that nonsymmetric Macdonald polynomials specialize to key polynomials. This was proved first by Ion \cite{Ion03}.

\begin{corollary}
  For a weak composition $a$, we have
  \begin{equation}
    \mac_a(X;0,0) = \key_a(X).
    \label{e:mac-key}
  \end{equation}
\end{corollary}

For our purposes, we will always take $t=0$ in which case the product becomes $1$, and the only terms that survive are those $T$ with $\coinv(T)=0$. This motivates the following definitions.

\begin{definition}
  The \emph{semi-standard key tabloids of shape $a$}, denoted by $\SSKD(a)$, are the non-attacking fillings with no co-inversion triples. The specialized nonsymmetric Macdonald polynomial $\mac_{a}(X;q,0)$ is given by
  \begin{equation}
    \mac_{a}(X;q,0) = \sum_{T \in \SSKD(a)} q^{\maj(T)} X^{\wt(T)} .
  \end{equation}
  \label{def:nsym-hall}
\end{definition}

\begin{figure}[ht]
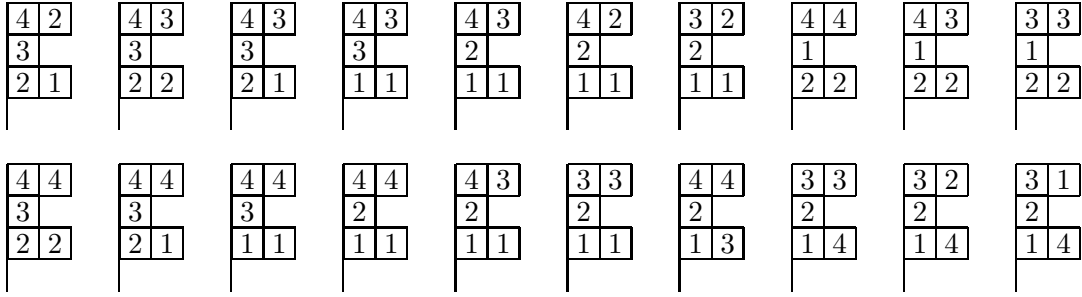

  \begin{displaymath}
    \begin{array}{c@{\hskip 1.5\cellsize}c@{\hskip 1.5\cellsize}c@{\hskip 1.5\cellsize}c@{\hskip 1.5\cellsize}c@{\hskip 1.5\cellsize}c@{\hskip 1.5\cellsize}c@{\hskip 1.5\cellsize}c@{\hskip 1.5\cellsize}c@{\hskip 1.5\cellsize}c}
      \vline\tableau{4 & 2 \\ 3 \\ 2 & 1 \\ & } & 
      \vline\tableau{4 & 3 \\ 3 \\ 2 & 2 \\ & } & 
      \vline\tableau{4 & 3 \\ 3 \\ 2 & 1 \\ & } & 
      \vline\tableau{4 & 3 \\ 3 \\ 1 & 1 \\ & } & 
      \vline\tableau{4 & 3 \\ 2 \\ 1 & 1 \\ & } & 
      \vline\tableau{4 & 2 \\ 2 \\ 1 & 1 \\ & } & 
      \vline\tableau{3 & 2 \\ 2 \\ 1 & 1 \\ & } & 
      \vline\tableau{4 & 4 \\ 1 \\ 2 & 2 \\ & } &
      \vline\tableau{4 & 3 \\ 1 \\ 2 & 2 \\ & } & 
      \vline\tableau{3 & 3 \\ 1 \\ 2 & 2 \\ & } \\ \\ 
      \vline\tableau{4 & 4 \\ 3 \\ 2 & 2 \\ & } & 
      \vline\tableau{4 & 4 \\ 3 \\ 2 & 1 \\ & } & 
      \vline\tableau{4 & 4 \\ 3 \\ 1 & 1 \\ & } & 
      \vline\tableau{4 & 4 \\ 2 \\ 1 & 1 \\ & } & 
      \vline\tableau{4 & 3 \\ 2 \\ 1 & 1 \\ & } & 
      \vline\tableau{3 & 3 \\ 2 \\ 1 & 1 \\ & } &
      \vline\tableau{4 & 4 \\ 2 \\ 1 & 3 \\ & } & 
      \vline\tableau{3 & 3 \\ 2 \\ 1 & 4 \\ & } & 
      \vline\tableau{3 & 2 \\ 2 \\ 1 & 4 \\ & } & 
      \vline\tableau{3 & 1 \\ 2 \\ 1 & 4 \\ & } 
    \end{array}
  \end{displaymath}
  \caption{\label{fig:SSKD}The semi-standard key tabloids of shape $(0,2,1,2)$.}
\end{figure}

For example, the $20$ semi-standard key tabloids of shape $(0,2,1,2)$ are given in Figure~\ref{fig:SSKD}. By Proposition~\ref{prop:SSKT-D}, every semi-standard key tableau is a semi-standard key tabloid, as evidenced by the fact that the first $16$ tabloids in Figure~\ref{fig:SSKD} are the $\SSKT(0,2,1,2)$ from Figure~\ref{fig:SSKT}. 

Our next task is to turn this monomial expansion into a fundamental slide expansion. To that end, we have the following definition. Note that we do not assume that the filling is non-attacking.

A \emph{standard key tabloid} is a bijective filling of a key diagram with no co-inversion triples. For example, there are ten standard key tabloids of shape $(0,2,1,2)$ as shown in Figure~\ref{fig:SKD}. Note that the first five are the standard key tableaux of shape $(0,2,1,2)$ shown in Figure~\ref{fig:SKT}.

\begin{figure}[ht]
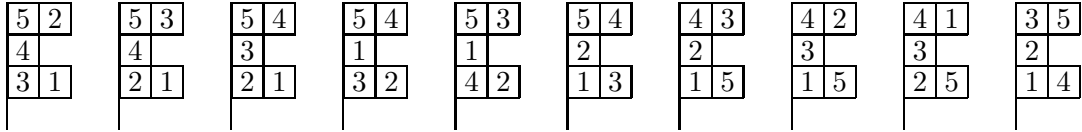

  \begin{displaymath}
    \begin{array}{c@{\hskip 1.5\cellsize}c@{\hskip 1.5\cellsize}c@{\hskip 1.5\cellsize}c@{\hskip 1.5\cellsize}c@{\hskip 1.5\cellsize}c@{\hskip 1.5\cellsize}c@{\hskip 1.5\cellsize}c@{\hskip 1.5\cellsize}c@{\hskip 1.5\cellsize}c}
      \vline\tableau{5 & 2 \\ 4 \\ 3 & 1 \\ & } &
      \vline\tableau{5 & 3 \\ 4 \\ 2 & 1 \\ & } &
      \vline\tableau{5 & 4 \\ 3 \\ 2 & 1 \\ & } &
      \vline\tableau{5 & 4 \\ 1 \\ 3 & 2 \\ & } &
      \vline\tableau{5 & 3 \\ 1 \\ 4 & 2 \\ & } &
      \vline\tableau{5 & 4 \\ 2 \\ 1 & 3 \\ & } & 
      \vline\tableau{4 & 3 \\ 2 \\ 1 & 5 \\ & } & 
      \vline\tableau{4 & 2 \\ 3 \\ 1 & 5 \\ & } & 
      \vline\tableau{4 & 1 \\ 3 \\ 2 & 5 \\ & } & 
      \vline\tableau{3 & 5 \\ 2 \\ 1 & 4 \\ & }
    \end{array}
  \end{displaymath}
  \caption{\label{fig:SKD}The standard key tabloids of shape $(0,2,1,2)$.}
\end{figure}

The following characterization follows from the proof of Proposition~\ref{prop:SSKT-D}.

\begin{proposition}
  The set of standard key tabloids $T$ of shape $a$ with $\maj(T)=0$ is $\SKT(a)$.
  \label{prop:SKT-D}  
\end{proposition}

We use standard objects to collect together terms for the fundamental slide expansion. Using Definition~\ref{def:weak-des}, we may associate a weak descent composition to each standard key tabloid. For example, latter $5$ standard key tabloids of shape $(0,2,1,2)$ shown in Figure~\ref{fig:SKD} have weak descent compositions $(1,1,1,2)$, $(1,1,2,1)$, $(1,2,1,1)$, $(2,1,1,1)$, $\varnothing$, respectively. Correspondingly, we have
\begin{eqnarray*}
  \mac_{(0,2,1,2)}(X;q,0) & = & \fund_{(1,2,1,1)} + \fund_{(0,2,2,1)} + \fund_{(0,2,1,2)} + \fund_{(1,2,0,2)}  \\
  & & q \left( \fund_{(1,1,1,2)} + \fund_{(1,1,2,1)} + \fund_{(1,2,1,1)} + \fund_{(2,1,1,1)} \right) .
\end{eqnarray*}
We prove that this holds in general, that is, that specialized nonsymmetric Macdonald polynomials are the fundamental slide generating polynomials of standard key tabloids.

\begin{theorem}
  The specialized nonsymmetric Macdonald polynomial $\mac_{a}(X;q,0)$ is given by
  \begin{equation}
    \mac_{a}(X;q,0) = \sum_{T \in \SKD(a)} q^{\maj(T)} \fund_{\des(T)}(X) ,
  \end{equation}
  where $\SKD(a)$ denotes the standard key tabloids of shape $a$.
  \label{thm:hall-slide}
\end{theorem}

\begin{proof}
  Reversing the proof of Proposition~\ref{prop:key-std}, define a \emph{standardization map}, denoted by $\std$, from semi-standard key tabloids to standard key tabloids as follows. Given $T \in \SSKD(a)$, relabel the cells of $T$ from $1$ to $n$ in the following order: for $k$ from $1$ to $n$, relabel cells labeled $k$ from left to right in $T$. The non-attacking condition ensures that columns have distinct values, therefore this is well-defined and necessarily results in a filling with distinct column entries. Moreover, by reading left to right beginning with the smallest entries, we ensure that the entry of a cell $c$ is strictly greater than the entry of the cell to its left in $T$ if and only the entry of a is strictly greater than the entry of the cell to its left in $\std(T)$. In particular, $\maj(\std(T)) = \maj(T)$. Given any triple, this observation together with distinct column values ensures that the entries form a co-inversion triple for $T$ if and only if they do for $\std(T)$, and thus $\std(T)\in\SKD(a)$. Exactly as in the proof of Proposition~\ref{prop:key-std}, we conclude that for any $S \in \SKD(a)$, we have
  \begin{displaymath}
    \sum_{T \in \std^{-1}(S)} q^{\maj(T)} X^{\wt(T)} = q^{\maj(S)}\fund_{\des(S)}(X),
  \end{displaymath}
  from which the result follows. 
\end{proof}

%
\section{Weak dual equivalence}
%
\label{sec:dual}

Analogous to our use of standardization in Theorem~\ref{thm:hall-slide} to group together monomials into fundamental slide polynomials, we use weak dual equivalence to group together fundamental slide polynomials into key polynomials.

Weak dual equivalence, introduced by Assaf in \cite{Ass-1}, provides a general framework for proving that a positive sum of fundamental slide polynomials is key positive. This generalizes dual equivalence \cite{Ass07,Ass15} which provides the analogous framework to prove Schur positivity of a function expressed in terms of Gessel's fundamental quasisymmetric functions. 

Given a set of combinatorial objects $\mathcal{A}$ endowed with a notion of weak descents, we consider the fundamental slide generating polynomial for $\mathcal{A}$ given by
\begin{displaymath}
  \sum_{T \in \mathcal{A}} \fund_{\des(T)} .
\end{displaymath}
Weak dual equivalence collects together terms into equivalence classes, each of which is a single key polynomial. We recall the main definitions and theorems from \cite{Ass-1}.

\begin{definition}[\cite{Ass-1}]
  Let $\mathcal{A}$ be a finite set, and let $\des$ be a map from $\mathcal{A}$ to weak compositions of $n$. A \emph{weak dual equivalence for $(\mathcal{A},\des)$} is a family of involutions $\{\psi_i\}_{1<i<n}$ on $\mathcal{A}$ such that
  \renewcommand{\theenumi}{\roman{enumi}}
  \begin{enumerate}
  \item For all $i-h \leq 3$ and all $T \in \mathcal{A}$, there exists a weak composition $a$ of $i-h+3$ such that
    \[ \sum_{U \in [T]_{(h,i)}} \fund_{\des_{(h-1,i+1)}(U)} = \key_{a}, \]
    where $[T]_{(h,i)}$ is the equivalence class generated by $\psi_h,\ldots,\psi_i$, and $\des_{(h,i)}(T)$ deletes the first $h-1$ and last $n-i$ nonzero parts from $\des(T)$.
    
  \item For all $|i-j| \geq 3$ and all $T \in\mathcal{A}$, we have $\psi_{j} \psi_{i}(T) = \psi_{i} \psi_{j}(T)$.

  \end{enumerate}

  \label{def:deg-weak}
\end{definition}

That is, a weak dual equivalence is a family of involutions for which equivalence classes of degree up to $6$ correspond to key polynomials and which commute when indices are far apart. 

The key polynomial $\key_a$ is \emph{$\fund$-stable} if both $\key_a$ and $\key_{0^m \times a}$ have the same number of terms in their fundamental slide expansions for any (equivalently some) $m>0$.

\begin{theorem}[\cite{Ass-1}]
  Let $\mathcal{A}$ be a set of combinatorial objects for which $\des$ is never $\varnothing$ (i.e. $\mathcal{A}$ has no virtual elements). If $\{\psi_i\}$ is a stable weak dual equivalence for $(\mathcal{A},\des)$, and $U \in \mathcal{A}$, then
  \begin{equation}
    \sum_{T \in [U]} \fund_{\des(T)} \ = \ \key_{a}
  \end{equation}
  for some $a$. In particular, the fundamental slide generating polynomial for $\mathcal{A}$ is key positive.
  \label{thm:positivity-key}
\end{theorem}

While the non-virtual condition might appear restrictive, in practice it may be ignore whenever the family of polynomials being considered behaves well under stabilization, which is the case for nonsymmetric Macdonald polynomials.

We recall the weak dual equivalence involutions for standard key tableaux defined in \cite{Ass-1}, and apply them to tabloids as well. For this definition, column reading order begins in the leftmost column, reading bottom to top, then continuing right.

\begin{definition}
  Given a bijective filling $T$ of a key diagram, define $\psi_i(T)$ as follows. Let $b,c,d$ be the cells with entries $i-1,i,i+1$ taken in column reading order. Then
  \begin{equation}
    \psi_i (T) = \left\{ \begin{array}{rl} 
      T & \mbox{if $c$ has entry $i$}, \\
      \braid_{i} (T) & \mbox{else if $b,d$ are attacking or row adjacent} , \\
      \swap_{i-1}(T) & \mbox{else if $c$ has entry $i+1$} , \\
      \swap_{i}  (T) & \mbox{else if $c$ has entry $i-1$} , 
    \end{array} \right.
  \end{equation}
  where $\braid_{j}$ cycles $j-1,j,j+1$ so that $j$ is not in position $c$ and $\swap_j$ interchanges $j$ and $j+1$.
  \label{def:deg-key}
\end{definition}

\begin{remark}
  In \cite{Ass-1}, the involutions on standard key tableaux apply $\braid_i$ precisely when $b,d$ are in the same row and $c$ is not. However, when $T$ is a standard key tableau, since $b$ and $d$ differ by $1$, they cannot be attacking, so they must be in the same row. Conversely, since rows weakly decrease, if $b$ and $d$ are in the same row then they must be adjacent in their row. Therefore these definitions agree when $T$ is a standard key tableau.
\end{remark}

For examples of $\braid_{i}$, $\swap_{i}$, and $\psi_i$ on standard key tabloids, see Figure~\ref{fig:weak-de}. Notice that these involutions exactly groups terms into equivalence classes corresponding to key polynomials, 
\[ \mac_{(0,2,1,2)}(X;q,0) = \key_{(0,2,1,2)}(X) + q \key_{(1,1,1,2)}(X) .\]

\begin{figure}[ht]
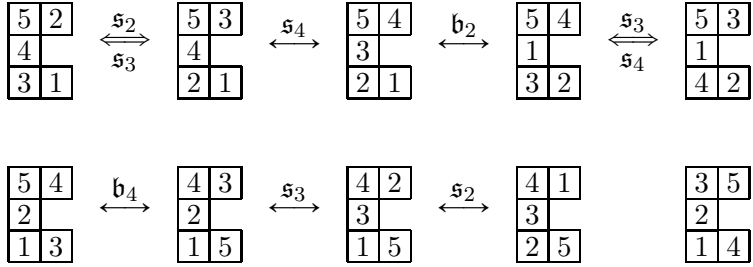

  \begin{displaymath}
    \begin{array}{ccccccccc}
      \tableau{5 & 2 \\ 4 \\ 3 & 1} & \raisebox{-\cellsize}{$\stackrel{\displaystyle\stackrel{\displaystyle\swap_2}{\Longleftrightarrow}}{\swap_3}$}  &
      \tableau{5 & 3 \\ 4 \\ 2 & 1} & \raisebox{-0.5\cellsize}{$\stackrel{\displaystyle\swap_{4}} \longleftrightarrow$} &
      \tableau{5 & 4 \\ 3 \\ 2 & 1} & \raisebox{-0.5\cellsize}{$\stackrel{\displaystyle\braid_{2}} \longleftrightarrow$} &
      \tableau{5 & 4 \\ 1 \\ 3 & 2} & \raisebox{-\cellsize}{$\stackrel{\displaystyle\stackrel{\displaystyle\swap_3}{\Longleftrightarrow}}{\swap_4}$}  &
      \tableau{5 & 3 \\ 1 \\ 4 & 2} \\ \\ \\
      \tableau{5 & 4 \\ 2 \\ 1 & 3} & \raisebox{-0.5\cellsize}{$\stackrel{\displaystyle\braid_{4}}\longleftrightarrow$} &
      \tableau{4 & 3 \\ 2 \\ 1 & 5} & \raisebox{-0.5\cellsize}{$\stackrel{\displaystyle\swap_{3}} \longleftrightarrow$} &
      \tableau{4 & 2 \\ 3 \\ 1 & 5} & \raisebox{-0.5\cellsize}{$\stackrel{\displaystyle\swap_{2}} \longleftrightarrow$} &
      \tableau{4 & 1 \\ 3 \\ 2 & 5} & &
      \tableau{3 & 5 \\ 2 \\ 1 & 4}
    \end{array}
  \end{displaymath}
  \caption{\label{fig:weak-de}Weak dual equivalence on the standard key tabloids of shape $(2,1,2)$}
\end{figure}

Before proving that these involutions give a weak dual equivalence for $\SKT(a)$, we note that they are well-defined and preserve the $\maj$ and $\coinv$ statistics. In fact, they preserve the set of cells for which the value exceeds that of the value to its left (or the row index for cells in the first column). 

\begin{lemma}
  Given a standard filling $T$ of a key diagram, we have $\maj(T) = \maj(\psi_i(T))$ and $\coinv(T) = \coinv(\psi_i(T))$.
  \label{lem:stats}
\end{lemma}

\begin{proof}
  Let $T$ be a standard filling of a key diagram. First we claim that for any cell $c$ of the key diagram for $a$, the entry of $c$ is greater than that of the cell to its left in $T$ if and only if the same holds for $\psi_i(T)$. If $\psi_i$ acts by swapping $i$ and $i\pm 1$, then the swapped entries are not row adjacent. For any $k \neq i, i\pm1$, we have $k>i$ if and only if $k> i \pm 1$, so the set of cells whose entry exceeds that of the cell to its left is preserved. If $\psi_i$ acts by $\braid_i$, for any $k \neq i-1,i,i+1$, we have $k>i+1$ or $k < i-1$, so the set of cells whose entry exceeds that of the cell to its left is also preserved unless, perhaps, two of $i-1,i,i+1$ are row adjacent. In this case, setting $i=2$ for notational convenience, we must have one of the eight situations depicted in Figure~\ref{fig:weak-triple} since the other four possible arrangements have $c=2$. Therefore $\braid_i$ acts as depicted in Figure~\ref{fig:weak-triple}, resolving this case.
  
  \begin{figure}[ht]
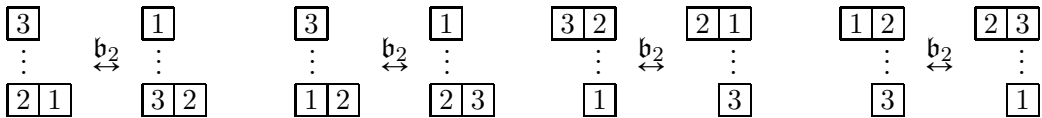

    \begin{displaymath}
      \begin{array}{l}
        \tableau{ 3 } \\[-0.5\cellsize] \hspace{0.4\cellsize} \vdots \\ \tableau{ 2 & 1 } 
      \end{array}
      \stackrel{\displaystyle\braid_2}{\leftrightarrow}
      \begin{array}{l}
        \tableau{ 1 } \\[-0.5\cellsize] \hspace{0.4\cellsize} \vdots \\ \tableau{ 3 & 2 } 
      \end{array}      
      \hspace{2\cellsize}
      \begin{array}{l}
        \tableau{ 3 } \\[-0.5\cellsize] \hspace{0.4\cellsize} \vdots \\ \tableau{ 1 & 2 } 
      \end{array}
      \stackrel{\displaystyle\braid_2}{\leftrightarrow}
      \begin{array}{l}
        \tableau{ 1 } \\[-0.5\cellsize] \hspace{0.4\cellsize} \vdots \\ \tableau{ 2 & 3 } 
      \end{array}      
      \hspace{\cellsize}
      \begin{array}{r}
        \tableau{ 3 & 2 } \\[-0.5\cellsize] \vdots \hspace{0.4\cellsize} \\ \tableau{ 1 } 
      \end{array}
      \stackrel{\displaystyle\braid_2}{\leftrightarrow}
      \begin{array}{r}
        \tableau{ 2 & 1 } \\[-0.5\cellsize] \vdots \hspace{0.4\cellsize} \\ \tableau{ 3 } 
      \end{array}
      \hspace{2\cellsize}
      \begin{array}{r}
        \tableau{ 1 & 2 } \\[-0.5\cellsize] \vdots \hspace{0.4\cellsize} \\ \tableau{ 3 } 
      \end{array}
      \stackrel{\displaystyle\braid_2}{\leftrightarrow}
      \begin{array}{r}
        \tableau{ 2 & 3 } \\[-0.5\cellsize] \vdots \hspace{0.4\cellsize} \\ \tableau{ 1 } 
      \end{array}
    \end{displaymath}
    \caption{\label{fig:weak-triple}Possible instances of $\braid_i$ on standard key tabloids when two of $1,2,3$ are adjacent.}
  \end{figure}

  Second, we claim that $\coinv(\psi_i(T)) = \coinv(T)$. To prove this, we appeal to an alternative characterization of the co-inversion number given in \cite{HHL08} as
  \begin{equation}
    \coinv(T) = \sum_{c} \mathrm{arm}(c) + \sum_{\mathrm{left}(c)<c} \mathrm{arm}(c) + \#\{(i<j) \mid a_i \leq a_j\} - \#\{(c<d) \mid d \mbox{ attacks } c\},
    \label{e:inv}
  \end{equation}
  where $\mathrm{left}(c)$ denotes the cell to the left of $c$ or the row index if $c$ is in the first column, $\mathrm{arm}(c)$ is the number of cells below $c$ in a weakly shorter row or above $c$ and one column to the left in a strictly shorter row, and $d$ attacks $c$ whenever the cells are attacking and $d$ occurs later in the column reading order. From this alternative characterization, it is enough to show that the number of these latter \emph{inversion pairs} is preserved since the other terms depend only on the shape or on the set of cells $c$ whose entry is greater than that of the entry to its left, which we have shown already is preserved. If $\psi_i$ acts by swapping $i$ with $i\pm 1$, then these two entries cannot be attacking, so by the same reasoning as before, the set of inversion pairs is preserved. Inspecting Figure~\ref{fig:weak-triple}, where now the corner cell could move up for the four left cases or down for the four right cases, shows that the number of inversion pairs, though not the set, is preserved by $\braid_i$.
\end{proof}

In particular, since $\coinv(T) = \coinv(\psi_i(T))$, Lemma~\ref{lem:stats} shows that the involutions $\psi_i$ are well-defined on standard key tabloids. Therefore we may now prove that they give a weak dual equivalence whenever all standard key tabloids are non-virtual. 

\begin{theorem}
  For a weak composition $a$ such that $\SKD(a)$ has no virtual elements, the maps $\{\psi_i\}$ on $\SKD(a)$ give a stable weak dual equivalence for $(\SKD(a),\des)$.
  \label{thm:weak-de}
\end{theorem}

\begin{proof}
  The action of $\psi_i$ on $T\in\SKT(a)$ is completely determined by the positions $i+1,i,i-1$, and the relative positions of cells other than these remains unchanged under $\psi_i$. Therefore, if $|i-j|\geq 3$, then $\{i-1,i,i+1\}$ and $\{j-1,j,j+1\}$ are disjoint, the maps $\psi_i$ and $\psi_j$ commute. This establishes condition (ii) of Definition~\ref{def:deg-weak}.

  To prove condition (i) of Definition~\ref{def:deg-weak}, we must consider restricted equivalence classes under $d_h,\ldots,d_i$ for $i-h \leq 3$. Consider the case $i-h=0$. With notation as in Definition~\ref{def:deg-key}, if the entry of $c$ is $i$, then both or neither of $i-1,i$ is a descent, so the restricted weak descent composition flattens to $(1,1,1)$ or $(3)$, respectively. In either case, the corresponding key polynomial is a single fundamental slide polynomial, and so the generating polynomial of the equivalence class, which is a single fundamental slide polynomial, is also a single key polynomial. 

  If $\psi_i(T)=\braid_i(T)$, then we may assume $c$ has entry $i+1$ since this is true either for $T$ or for $\psi_i(T)$. The partitioning from Definition~\ref{def:weak-des} for $T$ has $i\!+\!1 | i i\!-\!1$, and that for $\braid_i(T)$ has $i\!+\!1 i | i\!-\!1$. Therefore $\des_{(i-1,i+1)}(T) = (0^m,2,0^n,1)$, where $n=0$ if and only if either $i+1$ sorts to the row immediately above $i$ and $i-1$ or if $t_j$ of the block containing $i+1$ is forced to be $t_{j+1}-1$. Either way, we have $\des_{(i-1,i+1)}(\braid_i(T)) = (0^{m-1},1,2)$, and so the equivalence class has generating polynomial $\fund_{(0^{m},2,0^{n},1)} + \fund_{(0^{m-1},1,2)} = \key_{(0^{m},2,0^{n},1)}$.

  If $\psi_i(T)=\swap_{i-1}(T)$, then $c=i+1$, and we may assume $T$ has $i-1$ left of $i$ since this is true either for $T$ or for $\psi_i(T)$. Here, the partitioning for $T$ is $i\!+\!1 i | i\!-\!1$, and that of $\swap_{i-1}(T)$ is $i\!+\!1 | i i\!-\!1$. Moreover, in this case, for both $T$ and $\swap_{i-1}(T)$, both blocks of the run decomposition must be forced to have $t_j = t_{j+1}-1$ since there is an entry larger than $i$ left of $i$ and below $i-1$. Therefore $\des_{(i-1,i+1)}(T) = (0^m,1,2)$ and $\des_{(i-1,i+1)}(\swap_{i-1}(T)) = (0^{m},2,1)$, and so the equivalence class corresponds to the polynomial $\fund_{(0^{m},1,2)} + \fund_{(0^{m},2,1)} = \key_{(0^{m},1,2)}$. The argument for $\psi_i(T)=\swap_{i}(T)$ is completely analogous.
  
  One can either carry out similar analyses for the cases $i-h=1,2,3$, or, to avoid tedium, since the action of $\psi_i$ is determined by relative positions of these cells based on their rows and columns, there is a reasonably finite number configurations to check by computer. 
\end{proof}

As detailed in \cite{Ass-1}, a stable weak dual equivalence for $(\mathcal{A},\des)$ implicitly gives a $\des$-preserving bijection between an equivalence class of $\mathcal{A}$ and $\SKT(a)$ for some $a$. Using this, we say that an object of $\mathcal{A}$ is \emph{Yamanouchi} if is maps under this implicit bijection to the standard key tableau whose row reading word is the reverse of the identity. This is the unique $T \in \SKT(a)$ for which $\des(T) = a$. That is, Yamanouchi elements provide canonical representatives for equivalence classes from which one can readily determine the generating polynomial.

Given a weak composition $a$, we have a natural bijection $\SKD(a) \stackrel{\sim}{\rightarrow} \SKD(0^m\times a)$ for any positive integer $m$. If $T\in\SKD(a)$ is non-virtual, then the weak descent composition of the image of $T$ is $0^m \times\des(T)$. 
We say that a standard key tabloid $T\in\SKD(a)$ is \emph{Yamanouchi} if the corresponding non-virtual element $\SKD(0^m\times a)$ is Yamanouchi for some $m>0$. With this definition in place, we now arrive at our main result.

\begin{theorem}
  The specialized nonsymmetric Macdonald polynomial $\mac_{a}(X;q,0)$ is given by
  \begin{equation}
    \mac_{a}(X;q,0) = \sum_{T \in \YKD(a)} q^{\maj(T)} \key_{\des(T)} .
  \end{equation}
  In particular, $\mac_{a}(X;q,0)$ is a positive graded sum of Demazure characters.
  \label{thm:HL-key}
\end{theorem}

\begin{proof}
  If $\SKD(a)$ has no virtual elements, then by Theorem~\ref{thm:weak-de} and Theorem~\ref{thm:positivity-key}, each equivalence class under $\{\psi_i\}$ has generating polynomial a single key polynomial. By Lemma~\ref{lem:stats}, the $\maj$ statistic is constant on each class, so it factors out. Combining this with Theorem~\ref{thm:hall-slide}, we have
  \begin{equation}
    \mac_{a}(X;q,0) = \sum_{T \in \SKD(a)} q^{\maj(T)} \fund_{\des(T)} = \sum_{S \in \YKD(a)} q^{\maj(S)} \key_{\des(S)}.
  \end{equation}
  If $\SKD(a)$ has virtual elements, let $m>0$ be any integer such that $\SKD(0^m\times a)$ does not. The implicit bijection with standard key tableaux provided by the weak dual equivalence commutes with this padding. Let $\Psi$ be the induced $\des$-preserving map from $\SKD(0^{m}\times a)$ to $\SKT$. For any non-virtual elements of $\SKD(a)$, the corresponding standard key tabloids for $\SKD(0^{m}\times a)$ are precisely those that map to some standard key tableau with weak descent composition $0^{m}\times a$. Given any weak dual equivalence class, we may pull back both the standard key tabloids in $\SKD(0^{m}\times a)$ that are non-virtual in $\SKD(a)$ and those standard key tableaux that have at least $m$ leading $0$'s. This gives a $\des$-preserving bijection between non-virtual elements, so they must have the same generating polynomial. Hence $\mac_a(X;q,0)$ is also key positive with the same leading terms.
\end{proof}

%
\section{Symmetric polynomials}
%
\label{sec:mac}

In this section we relate the combinatorics for nonsymmetric Macdonald polynomials with that for the symmetric case, so that we now let $X$ denote the infinite set of variables $x_1,x_2,\ldots$.

Gessel \cite{Ges84} introduced \emph{quasisymmetric functions}, functions that are invariant under any sliding of the variables. Quasisymmetric functions are naturally index by \emph{compositions}, i.e. sequences of positive integers. Gessel's \emph{fundamental quasisymmetric function} is given by
\begin{equation}
  F_{\alpha}(X) = \sum_{\mathrm{flat}(b) \ \mathrm{refines} \ \alpha} X^{b}.
  \label{e:fund}
\end{equation}

Assaf and Searles \cite{AS17} showed that fundamental slide polynomials stabilize to fundamental quasisymmetric functions, that is
\begin{equation}
  \lim_{m\rightarrow\infty} \fund_{(0^{m}\times a)} (X) = F_{\mathrm{flat}(a)} (X).
  \label{e:slide-lift}
\end{equation}

A \emph{partition} is a weakly decreasing sequence of nonnegative integers. The \emph{Young diagram} of a partition $\lambda$ has $\lambda_i$ cells left-justified in row $i$. For example, see Figure~\ref{fig:SYT}.

A \emph{Young tableau} is a filling of a partition diagram such that entries weakly increase along rows and strictly increase up columns. To emphasize the range of numbers used, we let $\SSYT_n(\lambda)$ denote the set of \emph{semi-standard Young tableaux} of shape $\lambda$ with entries at most $n$. A \emph{standard Young tableau} is a bijective filling, and so it necessarily also has strictly increasing rows. Denote the set of standard Young tableaux of shape $\lambda$ by $\SYT(\lambda)$. For examples, see Figure~\ref{fig:SYT}.

\begin{figure}[ht]
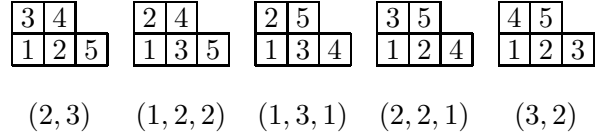

  \begin{displaymath}
    \begin{array}{ccccc}
      \tableau{3 & 4 \\ 1 & 2 & 5} & 
      \tableau{2 & 4 \\ 1 & 3 & 5} & 
      \tableau{2 & 5 \\ 1 & 3 & 4} & 
      \tableau{3 & 5 \\ 1 & 2 & 4} & 
      \tableau{4 & 5 \\ 1 & 2 & 3} \\ \\
      (2,3) & (1,2,2) & (1,3,1) & (2,2,1) & (3,2)
    \end{array}
  \end{displaymath}
  \caption{\label{fig:SYT}Standard Young tableaux of shape $(3,2)$ and their descent compositions.}
\end{figure}

The \emph{Schur polynomials} arise as characters for the irreducible representations of the general linear group, and also enjoy rich connections with geometry. The \emph{Schur function} $s_{\lambda}(X)$ is given by
\begin{equation}
  s_{\lambda}(X) = \sum_{T \in \SSYT(\lambda)} X^{\wt(T)}.
  \label{e:schur}
\end{equation}

Given a standard filling $T$ of a Young diagram, we may associate to it a composition called the \emph{descent composition of $T$}, denoted by $\Des(T)$, given by maximal length runs between descents, where $i$ is a \emph{descent} if $i+1$ lies weakly to its left. For examples, see Figure~\ref{fig:SYT}. 

\begin{proposition}[\cite{Ges84}]
  The Schur function $s_{\lambda}(X)$ is given by
  \begin{equation}
    s_{\lambda}(X) = \sum_{T \in \SYT(\lambda)} F_{\Des(T)}(X).
  \end{equation}
  \label{prop:gessel}
\end{proposition}

For example, from Figure~\ref{fig:SYT} we compute
\[ s_{(3,2)} = F_{(2,3)} + F_{(1,2,2)} + F_{(1,3,1)} + F_{(2,2,1)} + F_{(3,2)} . \]

\begin{proposition}[\cite{Ass-1}]
  Given a weak composition $a$ whose nonzero parts rearrange $\lambda$, the map $\phi$ that drops boxes to Young diagram shape, sorts columns into decreasing order (bottom to top) and replaces $i$ with $n-i+1$ gives a bijection between $\SKT(a)$ and $\SYT(\lambda)$. Moreover, for $T \in\SKT(a)$, we have $\Des(\phi(T)) = \mathrm{reverse}(\mathrm{flat}(\des(T)))$.
  \label{prop:stable}
\end{proposition}

From Proposition~\ref{prop:stable}, we have the following result that is implicit in work of Lascoux and Sch{\"u}tzenberger \cite{LS90}. For $a$ a weak composition of length $n$, let $\mathrm{sort}(a)$ be the partition with parts given by the nonzero entries of $a$. Then we have
\begin{equation}
  \lim_{m \rightarrow \infty} \key_{0^m \times a} (X) = s_{\mathrm{sort}(a)} (X).
  \label{e:key-stable}
\end{equation}

Haglund \cite{Hag04} discovered an elegant combinatorial formula for the monomial expansion of the transformed Macdonald functions that he proved together with Haiman and Loehr \cite{HHL05}. The combinatorial formula for transformed Macdonald functions is the sum over all fillings weighted by statistics that were the precursors of $\coinv$ and $\maj$ in the nonsymmetric case.

\begin{figure}[ht]
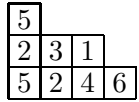

  \begin{displaymath}
    \tableau{ 5 \\ 2 & 3 & 1 \\ 5 & 2 & 4 & 6 }
  \end{displaymath}
  \caption{\label{fig:young-fill}A filling of the Young diagram for $(4,3,1)$}
\end{figure}

The \emph{leg} of a cell of a Young diagram is the number of cells weakly above it. Given a filling $T$, define $\comaj(T)$ to be the sum of the legs of all cells $c$ such that the entry in $c$ is weakly less than the entry immediately below it. For example, the filling in Figure~\ref{fig:young-fill} has $\comaj = 3$.

A \emph{triple} of a Young diagram is a collection of two or three cells with two in the same row and either these are in the bottom row or we include the cell immediately below the left cell. An \emph{inversion triple} is a triple with two cells and the larger to the left or with three cells oriented counter-clockwise. For an illustration, see Figure~\ref{fig:young-triple}. The \emph{inversion number} of a filling of a Young diagram is the number of inversion triples. For example, the filling in Figure~\ref{fig:young-fill} has $\inv = 3$.

\begin{figure}[ht]
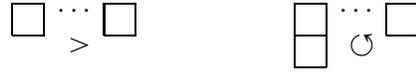

  \begin{displaymath}
    \tableau{ \ \\ } 
    \begin{array}{r}
      \cdots \\
      >
    \end{array}
    \tableau{ \ \\ & }
    \hspace{4\cellsize}
    \tableau{ \ \\ \ } 
    \begin{array}{r}
      \cdots \\
      \circlearrowleft
    \end{array}
    \tableau{ \ \\ & }
  \end{displaymath}
  \caption{\label{fig:young-triple}Inversion triples for Young diagrams.}
\end{figure}

\begin{proposition}
  The set of standard fillings of $\mu$ with $\comaj(T)=\inv(T)=0$ is $\SYT(mu)$.
  \label{prop:SSYT-D}
\end{proposition}

\begin{proof}
  The condition $\comaj=0$ forces columns to increase bottom to top. The lack of two cell inversion triples forces the bottom row to increase left to right, and this combined with increasing columns and no three cell inversion triples ensures all higher rows increase left to right as well.
\end{proof}

\begin{definition}[\cite{HHL05}]
  The Macdonald polynomial $\Mac_{\mu}(X_n;q,t)$ is given by
  \begin{equation}
    \Mac_{\mu}(X;q,t) = \sum_{T:\mu\stackrel{\sim}{\rightarrow}[n]} q^{\inv(T)} t^{\comaj(T)} F_{\Des(T)}(X),
  \end{equation}
  where the sum is over all standard fillings of $\mu$.
\end{definition}

In particular, from Proposition~\ref{prop:SSYT-D}, we have the following analog of \eqref{e:mac-key},
\begin{equation}
  \Mac_{\mu}(X;0,0) = s_{\mu}(X).
\end{equation}

There is a well-known $q,t$-symmetry for the transformed Macdonald polynomials that follows from Macdonald's original definition \cite{Mac88} when interpreted as triangularity conditions; it is
\begin{equation}
  \Mac_{\mu}(X;q,t) = \Mac_{\mu^{\prime}}(X;t,q),
\end{equation}
where $\mu^{\prime}$ is the \emph{conjugate of $\mu$} whose Young diagram is the transpose of the Young diagram for $\mu$.

Similar to the nonsymmetric case, we are interested in the surviving objects when $q=0$.

\begin{figure}[ht]
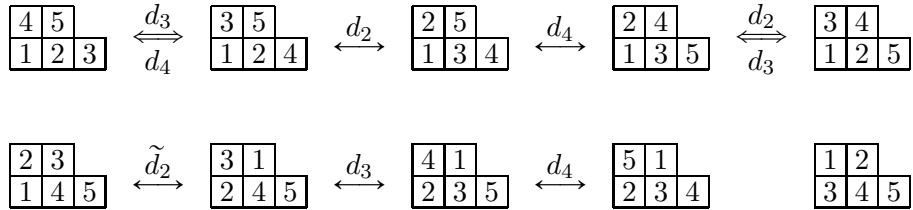

  \begin{displaymath}
    \begin{array}{ccccccccc}
      \tableau{4 & 5 \\ 1 & 2 & 3} & \raisebox{-\cellsize}{$\stackrel{\displaystyle\stackrel{\displaystyle d_3}{\Longleftrightarrow}}{d_4}$}  &
      \tableau{3 & 5 \\ 1 & 2 & 4} & \raisebox{-0.5\cellsize}{$\stackrel{\displaystyle d_{2}} \longleftrightarrow$} &
      \tableau{2 & 5 \\ 1 & 3 & 4} & \raisebox{-0.5\cellsize}{$\stackrel{\displaystyle d_{4}} \longleftrightarrow$} &
      \tableau{2 & 4 \\ 1 & 3 & 5} & \raisebox{-\cellsize}{$\stackrel{\displaystyle\stackrel{\displaystyle d_2}{\Longleftrightarrow}}{d_3}$}  &   
      \tableau{3 & 4 \\ 1 & 2 & 5} \\ \\ \\
      \tableau{2 & 3 \\ 1 & 4 & 5} & \raisebox{-0.5\cellsize}{$\stackrel{\displaystyle \widetilde{d}_{2}}\longleftrightarrow$} &
      \tableau{3 & 1 \\ 2 & 4 & 5} & \raisebox{-0.5\cellsize}{$\stackrel{\displaystyle d_{3}} \longleftrightarrow$} &
      \tableau{4 & 1 \\ 2 & 3 & 5} & \raisebox{-0.5\cellsize}{$\stackrel{\displaystyle d_{4}} \longleftrightarrow$} &
      \tableau{5 & 1 \\ 2 & 3 & 4} & &
      \tableau{1 & 2 \\ 3 & 4 & 5}
    \end{array}
  \end{displaymath}
  \caption{\label{fig:SYD}Standard Young tabloids of shape $(3,2)$}
\end{figure}

\begin{definition}
  The set of \emph{standard Young tabloids of shape $\mu$}, denoted by $\SYD(\mu)$, are bijective filling of $\mu$ with no inversion triples. The Hall--Littlewood polynomial $\Mac_{\mu}(X_n;0,t)$ is given by
  \begin{equation}
    \Mac_{\mu}(X;0,t) = \sum_{T \in \SYD(\mu)} t^{\comaj(T)} F_{\Des(T)}(X).
  \end{equation}
\end{definition}

For example, there are ten standard Young tabloids of shape $(3,2)$ as shown in Figure~\ref{fig:SYD}. From this we compute
\begin{eqnarray*}
  \Mac_{(3,2)}(X;0,t) & = & \left( F_{(2,3)} + F_{(1,2,2)} + F_{(1,3,1)} + F_{(2,2,1)} + F_{(3,2)} \right) \\
  & & t \left( F_{(1,4)} + F_{(2,3)} + F_{(3,2)} + F_{(4,1)} \right) + t^2 F_{(5)} .
\end{eqnarray*}

In fact, Hall-Littlewood polynomials are well known to be Schur positive; e.g. see \cite{Mac95}.  For example, from the previous computation we see
\[ \Mac_{(3,2)}(X;0,t) = s_{(3,2)}(X) + t s_{(4,1)}(X) + t^2 s_{(5)}(X) .\]

For comparison, using the stability of fundamental slide polynomials and key polynomials, we see from Figure~\ref{fig:weak-de} that
\[ \lim_{m \rightarrow\infty} \mac_{0^m \times (2,1,2)}(X;q,0) = s_{(2,2,1)}(X) + q s_{(2,1,1,1)}(X) + q^2 s_{(1,1,1,1,1)}(X). \]

Recall the well-known symmetric function involution $\omega$ defined by $\omega s_{\lambda} = s_{\lambda^{\prime}}$. Then we have,
\[ \lim_{m \rightarrow\infty} \mac_{0^m \times (2,1,2)}(X;q,0) = \omega \Mac_{(3,2)}(X;0,q) = \omega \Mac_{(2,2,1)}(X;q,0). \]
We give a bijective proof of this stability result in general.

\begin{theorem}
  For a weak composition $a$, we have
  \begin{equation}
    \lim_{m \rightarrow\infty} \mac_{0^m \times a}(X;q,0) = \omega \Mac_{\mathrm{sort}(a)^{\prime}}(X;0,q) = \omega \Mac_{\mathrm{sort}(a)}(X;q,0) .
  \end{equation}
  \label{thm:mac-stable}
\end{theorem}

\begin{proof}
  For each set partitioning of $[n]$ into the rows (respectively, columns) of a Young (respectively, key) diagram has a unique standard Young (respectively, key) tabloid with those row (respectively, column) entries \cite{Hag08}. Composing this with the map on entries that sends $i$ to $n-i+1$ gives a bijection, say $\theta$, between $\SKD(a)$ and $\SYD(\lambda)$ for any weak composition $a$ whose nonzero parts rearrange $\lambda$. We claim that this bijection commutes the with dual equivalence structures on the corresponding sets in the following sense. Recall the involutions $D_i$ on standard fillings of a Young diagram from \cite{Ass15} that preserve Haglund's $\inv$ and $\maj$ statistics. Then for $T\in\SKT(\lambda)$, we have $\theta(\psi_i(T)) = D_{n-i+1}(\theta(T))$; for example, compare Figures~\ref{fig:weak-de} and \ref{fig:SYD}. Moreover, the descent composition for $\SYD$ is with respect to the row reading word (left to right) and the weak descent composition for $\SKD$ is with respect to the column reading word (bottom to top). Identifying compositions of $n$ with subsets of $[n-1]$, for $T\in\SKD(\lambda)$, $\mathrm{flat}(\des(T))$ is the complement of $\Des(\theta(T))$. The theorem now follows from the stability of fundamental slide polynomials \eqref{e:slide-lift}. 
\end{proof}

For a final comparison, define the \emph{Kostka--Foulkes polynomial} $K_{\lambda,\mu}(t)$ by
\begin{equation}
  \Mac_{\mu}(X;0,t) = \sum_{\lambda} K_{\lambda,\mu}(t) s_{\lambda}(X).
\end{equation}
Then, considering the $30$ standard Young tabloids of shape $(2,2,1)$, exactly two of the dual equivalences classes have generating function $s_{(3,2)}(X)$, giving
\[ K_{(3,2),(2,2,1)}(t) = t + t^2 . \]

Define the \emph{nonsymmetric Kostka--Foulkes polynomial} $K_{a,b}(q)$ by
\begin{equation}
  \mac_{b}(X;q,0) = \sum_{a} K_{a,b}(q) \key_{a}(X).
\end{equation}
Then, considering the $30$ standard key tabloids of shape $(3,0,2)$, exactly two of the dual equivalences classes have generating polynomial that stabilizes to $s_{(2,2,1)}(X)$, giving
\[ K_{(2,1,2),(3,0,2)}(q) = q  \hspace{1em} \mbox{and} \hspace{1em} K_{(1,2,2),(3,0,2)}(q) = q^2 . \]

Therefore we may reformulate Theorem~\ref{thm:mac-stable} in terms of Kostka--Foulkes polynomials as follows.

\begin{corollary}
  Given a weak composition $b$ with column lengths $\mu$ such that $\SKT(b)$ has no virtual Yamanouchi elements, we have
  \begin{equation}
    K_{\lambda,\mu}(t) = \sum_{\mathrm{sort}(\mathrm{flat}(a)) = \lambda^{\prime}} K_{a,b}(t) .
  \end{equation}
\end{corollary}

It remains to be seen if the analogous statement holds for general types.

%
%

\bibliographystyle{amsalpha} 
\bibliography{hall.bib}

\end{document}